\documentclass[leqno,12pt]{amsart}
\usepackage{amsfonts}
\usepackage{amsmath,amssymb,amsthm}

\newcommand{\R}{\mathbb R}
\newcommand{\E}{\mathbb E}
\renewcommand{\span}{\mathrm{span}}
\newcommand{\tr}{\mathrm{tr}}
\newcommand{\ds}{\displaystyle}

\everymath{\displaystyle}

\newtheorem{thm}{Theorem}[section]

\theoremstyle{definition}

\theoremstyle{remark}
\newtheorem{rem}[thm]{Remark}

\begin{document}

\title[Meridian Surfaces with Parallel Normalized Mean Curvature Vector in $\E^4_2$] {Meridian Surfaces with Parallel Normalized Mean Curvature Vector Field in  Pseudo-Euclidean 4-space with Neutral Metric}

\author{ Bet\"{u}l Bulca and Velichka Milousheva}
\address{Uluda\u{g} University, Department of Mathematics, 16059 Bursa, Turkey}
\email{bbulca@uludag.edu.tr}
\address{Institute of Mathematics and Informatics, Bulgarian Academy of Sciences,
Acad. G. Bonchev Str. bl. 8, 1113, Sofia, Bulgaria;   "L. Karavelov"
Civil Engineering Higher School, 175 Suhodolska Str., 1373 Sofia,
Bulgaria} \email{vmil@math.bas.bg}


\begin{abstract}
We construct a special class of Lorentz surfaces in the
pseudo-Euclidean 4-space with neutral metric which are one-parameter systems of meridians of
 rotational hypersurfaces with timelike or spacelike axis and call them meridian surfaces.
 We give the complete classification of the meridian surfaces with parallel mean curvature vector field. We also classify the meridian surfaces with parallel normalized mean curvature vector. We show that in the family of the meridian surfaces there exist  Lorentz surfaces which have parallel normalized mean curvature vector field but not parallel mean curvature vector.
\end{abstract}

\maketitle

\section{Introduction}

A basic class of surfaces in Riemannian and pseudo-Riemannian geometry are surfaces with parallel mean curvature vector field, since they are critical points of some natural functionals and  play important role in differential geometry,  the theory of harmonic maps, as well as in physics.
The classification of surfaces with parallel mean curvature vector field in Riemannian space forms was given by Chen \cite{Chen1} and Yau  \cite{Yau}. Recently, spacelike surfaces with parallel mean
curvature vector field  in pseudo-Euclidean spaces with arbitrary codimension  were classified in \cite{Chen1-2} and  \cite{Chen1-3}.
Lorentz surfaces with parallel mean curvature vector field in arbitrary pseudo-Euclidean space $\E^m_s$ are studied in \cite{Chen-KJM} and \cite{Fu-Hou}. A nice survey on classical and recent results on
submanifolds with parallel mean curvature vector in Riemannian manifolds
as well as in pseudo-Riemannian manifolds is presented in \cite{Chen-survey}.

The class of surfaces with parallel mean curvature vector field is naturally extended to the class of surfaces with parallel
normalized mean curvature vector field.
A submanifold in a Riemannian manifold is said to have parallel
normalized mean curvature vector field  if the mean curvature vector is non-zero and the
unit vector in the direction of the mean curvature vector is parallel in the normal
bundle  \cite{Chen-MM}.
It is well known that submanifolds with non-zero parallel mean curvature vector field also have parallel normalized mean curvature vector field.
 But the condition to have parallel normalized mean curvature vector field is  weaker than the condition to have parallel mean curvature vector field.
For example, every surface in the Euclidean 3-space has parallel normalized mean
curvature vector field but in the 4-dimensional Euclidean space, there exist abundant examples of surfaces
which lie fully in $\E^4$ with parallel normalized mean
curvature vector field, but not with parallel mean curvature vector field.
In \cite{Chen-MM} it is proved that every analytic surface  with parallel normalized mean curvature vector in the Euclidean space
$\E^m$ must either lie in a 4-dimensional space $\E^4$ or in a hypersphere of $\E^m$ as a minimal surface.

In the pseudo-Euclidean space with neutral metric $\E^4_2$ the study of Lorentz surfaces with  parallel normalized mean
curvature vector field, but not  parallel mean curvature vector field, is still an open problem.

In the present paper we construct special families of 2-dimensional Lorentz surfaces in  $\E^4_2$  which lie on rotational hypersurfaces with timelike or spacelike axis and call them
 meridian surfaces. Depending on the type of the spheres in $\E^3_1$ (resp. $\E^3_2$) and the casual character of the spherical curves, we distinguish three types of Lorentz meridian surfaces in $\E^4_2$. These surfaces are analogous to  the meridian surfaces in the Euclidean space $\E^4$ and the Minkowski space $\E^4_1$, which are defined and studied in \cite{GM2}, \cite{GM-BKMS}, and \cite{GM6}, \cite{GM-MC}, respectively.

In Theorems  \ref{T:parallel-H-a}, \ref{T:parallel-H-b}, and \ref{T:parallel-H-c}  we give the complete classification of all Lorentz meridian surfaces (of these three types) which have parallel mean curvature vector field. We also classify the meridian surfaces with parallel normalized mean curvature vector field (Theorems \ref{T:parallel-H0-a}, \ref{T:parallel-H0-b}, and \ref{T:parallel-H0-c}).
In the family of the meridian surfaces  we find examples of Lorentz surfaces in $\E^4_2$  which have parallel normalized mean curvature vector field but not parallel mean curvature vector field.

\section{Preliminaries}

 Let $\E^4_2$ be the pseudo-Euclidean 4-dimensional space  with the canonical pseudo-Euclidean   metric of index 2
given in local coordinates by $$\widetilde{g} = dx_1^2 + dx_2^2 - dx_3^2 - dx_4^2,$$
where $(x_1, x_2, x_3, x_4)$ is a rectangular coordinate system of $\E^4_2$. We denote by $ \langle ., . \rangle$ the indefinite inner scalar product with respect to  $\widetilde{g}$.
Since $\widetilde{g}$ is an indefinite
metric,  a vector $v \in \E^4_2$ can have one of
the three casual characters: it can be \textit{spacelike} if $\langle v, v \rangle >0$
or $v=0$, \textit{timelike} if $\langle v, v \rangle<0$, and \textit{null} (\textit{lightlike}) if $\langle v, v \rangle =0$ and $v\neq 0$.

We use the following denotations:
\begin{equation*}
\begin{array}{l}
\vspace{2mm}
\mathbb{S}^3_2(1) =\left\{V\in \E^4_2: \langle V, V \rangle =1 \right\}; \\
\vspace{2mm}
\mathbb{H}^3_1(-1) =\left\{ V\in \E^4_2: \langle V, V \rangle = -1\right\}.
\end{array}
\end{equation*}
The space $\mathbb{S}^3_2(1)$ is known as the de Sitter space, and the
space $\mathbb{H}^3_1(-1)$ is the hyperbolic space (or the anti-de Sitter space) \cite{O'N}.

Given a surface $M$ in $\E^4_2$, we denote by $g$ the induced metric of  $\widetilde{g}$ on $M$. A surface $M$ in $\E^4_2$
 is called \emph{Lorentz}  if the
induced  metric $g$ on $M$ is Lorentzian. Thus, at each point $p\in M$ we have the following decomposition
$$\E^4_2 = T_pM \oplus N_pM$$
with the property that the restriction of the metric
onto the tangent space $T_pM$ is of
signature $(1,1)$, and the restriction of the metric onto the normal space $N_pM$ is of signature $(1,1)$.

Denote by $\nabla $ and $\widetilde{\nabla }$ the Levi-Civita connections
of $M$ and $\mathbb{E}_{2}^{4}$, respectively. For any vector fields $x,y$
tangent to $M$ the Gauss formula is given by
\begin{equation*}
\widetilde{\nabla}_xy=\nabla_xy + h(x,y)
\end{equation*}
where $h$ is the second fundamental form of $M$. Let $D$ denotes the normal
connection on the normal bundle of $M$. Then for any normal vector
field $\xi $ and any tangent vector field $x$  the Weingarten formula
is given by
\begin{equation*}
\widetilde{\nabla}_x \xi = - A_{\xi }x + D_x\xi,
\end{equation*}
where $A_{\xi }$ is the shape operator with respect to $\xi$.

The mean curvature vector  field $H$ of $M$ in $\E^4_2$
is defined as $H = \ds{\frac{1}{2}\,  \tr\, h}$.
A surface $M$  is called \emph{minimal} if its mean curvature vector vanishes identically, i.e. $H =0$.
A natural extension of minimal surfaces are quasi-minimal surfaces.
A surface $M$  is called \emph{quasi-minimal} (or \textit{pseudo-minimal}) if its
mean curvature vector is lightlike at each point, i.e. $H \neq 0$ and $\langle H, H \rangle =0$  \cite{Rosca}.

A normal vector field $\xi$ on $M$ is called \emph{parallel in the normal bundle} (or simply \emph{parallel}) if $D{\xi}=0$ holds identically \cite{Chen2}. A surface $M$ is said to have \emph{parallel mean curvature vector field} if its mean curvature vector $H$
satisfies $D H =0$ identically.

Surfaces for which the mean curvature vector field $H$ is non-zero, $\langle H, H \rangle \neq 0$, and  there exists a unit vector field $b$ in the direction of the mean curvature vector
 $H$, such that $b$ is parallel in the normal
bundle, are called surfaces with \textit{parallel normalized mean curvature vector field} \cite{Chen-MM}.
It is easy to see  that if $M$ is a surface  with non-zero parallel mean curvature vector field $H$ (i.e. $DH = 0$),
then $M$ is a surface with parallel normalized mean curvature vector field, but the converse is not true in general.
It is true only in the case $\Vert H \Vert = const$.

\section{Construction of meridian surfaces  in $\E^4_2$}

In \cite{GM2} G. Ganchev and the second author  constructed a family of surfaces lying on a
standard rotational hypersurface in the Euclidean 4-space $\mathbb{E}^{4}$. These surfaces are one-parameter systems of meridians of
the rotational hypersurface, that is why they called them \textit{meridian surfaces}. In \cite{GM2}  and \cite{GM-BKMS}
they gave the classification of  the meridian surfaces with constant Gauss curvature, with constant mean curvature, Chen meridian surfaces and  meridian surfaces with parallel normal bundle. The meridian surfaces in $\E^4$ with pointwise 1-type Gauss map are classified in
\cite{ABM}.
In \cite{GM6} and \cite{GM-MC} they used the  idea from the Euclidean case
to construct special families of two-dimensional spacelike surfaces lying on
rotational hypersurfaces in $\mathbb{E}_{1}^{4}$ with timelike or spacelike
axis and gave the classification of meridian surfaces from  the same basic classes.
The meridian surfaces in $\E^4_1$ with pointwise 1-type Gauss map are classified in
\cite{AM}.

Following the idea from the Euclidean and Minkowski spaces, in \cite{BM} we constructed Lorentz meridian surfaces in the pseudo-Euclidean 4-space $\E^4_2$  as one-parameter systems of meridians of rotational hypersurfaces with timelike or
spacelike axis. We gave the classification of quasi-minimal meridian  surfaces and meridian surfaces with constant mean curvature (CMC-surfaces).
Here we shall present briefly the construction.

\subsection{Lorentz meridian surfaces lying on a rotational hypersurface with timelike axis}

Let $Oe_1 e_2 e_3 e_4$ be a fixed orthonormal coordinate system in $\E^4_2$, i.e. $\langle e_1,e_1 \rangle
=\langle e_2,e_2 \rangle =1, \langle e_3,e_3 \rangle = \langle e_4,e_4 \rangle = -1$. We shall consider a
 rotational hypersurface with timelike axis $Oe_4$. Similarly, one can consider a rotational hypersurface with axis $Oe_3$.

In the Minkowski space $\E^3_1 = \span\left \{ e_{1},e_{2},e_{3}\right \}$ there are two types of  two-dimensional spheres, namely the  pseudo-sphere  $\mathbb{S}^2_1(1) =\left\{V\in \E^3_1: \langle V, V \rangle = 1\right \}$, i.e. the de Sitter space, and the pseudo-hyperbolic sphere $\mathbb{H}^2_1(-1)=\left\{V \in  \E^3_1: \langle V, V \rangle = - 1\right \}$, i.e. the anti-de Sitter space. So, we can consider two types of rotational hypersurfaces about the axis  $Oe_4$.

Let  $f=f(u)$, $g=g(u)$ be smooth functions, defined in an interval $I\subset \R$. The first type rotational
hypersurface $\mathcal{M}^I$ in $\E^4_2$, obtained by the rotation of the
meridian curve $m:u\rightarrow (f(u),g(u))$ about the $Oe_4$-axis, is
parametrized as follows:
\begin{equation*}
\mathcal{M}^I:Z(u,w^1,w^2) = f(u)\left(\cosh w^1 \cos w^2 \,e_1 +   \cosh w^1 \sin w^2 \,e_2 + \sinh w^1 \,e_3\right) + g(u) \,e_4.
\end{equation*}
Note that $l^I(w^1,w^2) = \cosh w^1 \cos w^2 \,e_1 + \cosh w^1 \sin w^2 \,e_2 + \sinh w^1 \,e_3$ is the unit position vector of the sphere
$\mathbb{S}^2_1(1)$ in $\E^3_1 = \span\left \{ e_{1},e_{2},e_{3}\right \}$ centered at the origin $O$. The parametrization of $\mathcal{M}^I$ can be written as:
\begin{equation*}
\mathcal{M}^I:Z(u,w^1,w^2) = f(u) l^I(w^1,w^2) + g(u) \,e_4.
\end{equation*}

The second type rotational
hypersurface $\mathcal{M}^{II}$ in $\E^4_2$, obtained by the rotation of the
meridian curve $m$ about the axis $Oe_4$, is given by the following
parametrization:
\begin{equation*}
\mathcal{M}^{II}:Z(u,w^1,w^2) = f(u)\left(\sinh w^1 \cos w^2 \,e_1 +  \sinh w^1 \sin w^2 \,e_2 + \cosh w^1 \,e_3\right) + g(u) \,e_4.
\end{equation*}
If we denote by  $l^{II}(w^1,w^2) = \sinh w^1 \cos w^2 \,e_1 +  \sinh w^1 \sin w^2 \,e_2 + \cosh w^1 \,e_3$ the unit position vector of the hyperbolic sphere
$\mathbb{H}^2_1(-1)$ in $\E^3_1 = \span\left \{ e_{1},e_{2},e_{3}\right \}$ centered at the origin $O$, then the parametrization of
$\mathcal{M}^{II}$ can be written as:
\begin{equation*}
\mathcal{M}^{II}:Z(u,w^1,w^2) = f(u) l^{II}(w^1,w^2) + g(u) \,e_4.
\end{equation*}

We shall construct Lorentz surfaces in $\E^4_2$ which are one-parameter systems of meridians of the hypersurface $\mathcal{M}^I$ or $\mathcal{M}^{II}$.

\vskip 3mm
\textbf{Meridian surfaces on  $\mathcal{M}^I$}:

Let $w^1 = w^1(v)$, $w^2=w^2(v), \,\, v \in J, \, J \subset \R$. Then  $c: l = l(v) =  l^I(w^1(v),w^2(v))$ is a smooth curve on  $\mathbb{S}^2_1(1)$.  We consider the two-dimensional
 surface $\mathcal{M}'$ lying on $\mathcal{M}^I$ and defined by:
\begin{equation}  \label{E:Eq-1}
\mathcal{M}': z(u,v) = f(u)\, l(v) + g(u) \,e_4, \quad u \in I, \, v \in J.
\end{equation}
The surface $\mathcal{M}'$, defined by \eqref{E:Eq-1}, is a
one-parameter system of meridians of $\mathcal{M}^I$, so we call
it a  \emph{meridian surface on $\mathcal{M}^I$}.

The tangent space of $\mathcal{M}'$ is spanned by the vector fields
$$z_u = f'(u) \,l(v) + g'(u) \,e_4; \qquad  z_v = f(u) \,l'(v),$$
so, the coefficients of the first fundamental form of $\mathcal{M}'$  are
$$E = \langle z_u, z_u \rangle =  f'^2 - g'^2; \quad F = \langle z_u, z_v \rangle = 0; \quad  G = \langle z_v, z_v \rangle = f^2 \langle l',l'\rangle.$$
Since we are interested in Lorentz surfaces, in the case   the spherical curve $c$ is spacelike, i.e. $\langle l',l'\rangle>0$, we take the meridian curve $m$ to be timelike,  i.e. $f'^2 - g'^2 <0$; and if
$c$ is timelike, i.e. $\langle l',l'\rangle<0$ , we take $m$ to be spacelike, i.e.  $f'^2 - g'^2 >0$.

\vskip 2mm
\textbf{Case (a)}: Let $\langle l',l'\rangle =1$, i.e. $c$ is spacelike.
We denote by  $t(v) = l'(v)$  the tangent vector field of $c$. Since $\langle t(v), t(v) \rangle = 1$,
 $\langle l(v), l(v) \rangle = 1$, and $\langle t(v), l(v) \rangle = 0$,
 there exists a unique (up to a sign)
 vector field $n(v)$, such that
$\{l(v), t(v), n(v)\}$ is an orthonormal frame field in $\E^3_1$ (note that $\langle n(v), n(v) \rangle = -1$). With respect to this
 frame field we have the following Frenet formulas of $c$ on $\mathbb{S}^2_1(1)$:
\begin{equation} \label{E:Eq-2}
\begin{array}{l}
\vspace{2mm}
l' = t;\\
\vspace{2mm}
t' = - \kappa \,n - l;\\
\vspace{2mm} n' = - \kappa \,t,
\end{array}
\end{equation}
 where $\kappa (v)= \langle t'(v), n(v) \rangle$ is the spherical curvature of $c$ on  $\mathbb{S}^2_1(1)$.

Without loss of generality we assume that $f'^2 - g'^2 = -1$. Then for
the coefficients of the first fundamental form we have $E = -1; \, F = 0; \, G = f^2(u)$.
Hence, in this case the meridian surface, defined by \eqref{E:Eq-1}, is a Lorentz surface in $\E^4_2$.
We denote this surface by $\mathcal{M}'_a$.

Now we consider the unit tangent vector fields  $X = z_u,\,\, Y = \ds{\frac{z_v}{f} = t}$,
which satisfy $\langle X,X \rangle = -1$,  $\langle Y,Y \rangle = 1$ and $\langle X,Y \rangle =0$, and the
following  normal vector fields:
\begin{equation} \label{E:Eq-2-2}
n_1 = n(v); \qquad n_2 = g'(u)\,l(v) + f'(u) \, e_4.
\end{equation}
Thus we obtain a frame field $\{X,Y, n_1, n_2\}$ of $\mathcal{M}'_a$, such that $\langle n_1, n_1 \rangle = -1$,
$\langle n_2, n_2 \rangle = 1$, $\langle n_1, n_2 \rangle =0$.

Taking into account \eqref{E:Eq-2} we get:
\begin{equation} \label{E:Eq-3}
\begin{array}{l}
\vspace{2mm}
h(X,X) = \kappa_m\,n_2; \\
\vspace{2mm}
h(X,Y) =  0; \\
\vspace{2mm}
h(Y,Y) = - \frac{\kappa}{f}\,n_1 - \frac{g'}{f} \, n_2,
\end{array}
\end{equation}
where $\kappa_m$ denotes the curvature of the meridian curve $m$, i.e.
$\kappa_m (u)= f'' g' - f' g''$.
Formulas \eqref{E:Eq-3}  and the equality $f'^2 - g'^2 = -1$  imply that the Gauss curvature $K$ and the normal mean curvature vector field $H$ of the meridian surface $\mathcal{M}'_a$ are given, respectively by
\begin{equation} \label{E:Eq-Ka}
K=\frac{f''}{f};
\end{equation}
\begin{equation} \label{E:Eq-Ha}
H=-\frac{\kappa}{2f} \,n_1 - \frac{f f''+(f')^2 + 1}{2f \sqrt{f'^2+1}} \,n_2.
\end{equation}

\vskip 2mm
\textbf{Case (b)}: Let $\langle l',l'\rangle = -1$, i.e. $c$ is timelike. In this case we assume that $f'^2 - g'^2 = 1$.
We denote  $t(v) = l'(v)$ and consider an orthonormal frame field $\{l(v), t(v), n(v)\}$ of $\E^3_1$, such that
$\langle l, l \rangle = 1$, $\langle t, t \rangle = -1$,
$\langle n, n \rangle = 1$. Then  we have the following Frenet formulas of $c$ on $\mathbb{S}^2_1(1)$:
\begin{equation} \label{E:Eq-2-b}
\begin{array}{l}
\vspace{2mm}
l' = t;\\
\vspace{2mm}
t' =  \kappa \,n + l;\\
\vspace{2mm} n' =  \kappa \,t,
\end{array}
\end{equation}
 where $\kappa (v)= \langle t'(v), n(v) \rangle$ is the spherical curvature of $c$ on  $\mathbb{S}^2_1(1)$.

In this case the coefficients of the first fundamental form are $E = 1; \, F = 0; \, G = -f^2(u)$.
We denote the meridian surface in this case  by $\mathcal{M}'_b$.

Again we consider the unit tangent vector fields  $X = z_u,\,\, Y = \ds{\frac{z_v}{f} = t}$,
which satisfy $\langle X,X \rangle = 1$,  $\langle Y,Y \rangle = -1$ and $\langle X,Y \rangle =0$, and the
following  normal vector fields:
$$n_1 = n(v); \qquad n_2 = g'(u)\,l(v) + f'(u) \, e_4,$$
 satisfying $\langle n_1, n_1 \rangle = 1$,
$\langle n_2, n_2 \rangle = -1$, $\langle n_1, n_2 \rangle =0$.

Using \eqref{E:Eq-2-b} we get:
\begin{equation} \label{E:Eq-3-b}
\begin{array}{l}
\vspace{2mm}
h(X,X) = \kappa_m\,n_2; \\
\vspace{2mm}
h(X,Y) =  0; \\
\vspace{2mm}
h(Y,Y) = \frac{\kappa}{f}\,n_1 -\frac{g'}{f} \, n_2,
\end{array}
\end{equation}
where $\kappa_m$ is the curvature of the meridian curve $m$, which in the case of a spacelike curve is given by the formula
$\kappa_m (u)=  f' g'' - f'' g'$.
Formulas \eqref{E:Eq-3-b}  and the equality $f'^2 - g'^2 = 1$  imply that
the Gauss curvature $K$ and the normal mean curvature vector field $H$ of the meridian surface $\mathcal{M}'_b$ are expressed  as follows:
\begin{equation*} \label{E:Eq-Kb}
K=-\frac{f''}{f};
\end{equation*}
\begin{equation} \label{E:Eq-Hb}
H=-\frac{\kappa}{2f} \,n_1 + \frac{f f''+(f')^2 - 1}{2f \sqrt{f'^2-1}} \,n_2.
\end{equation}

\vskip 3mm
\textbf{Meridian surfaces on  $\mathcal{M}^{II}$}:

Now we shall construct meridian surfaces lying on the rotational hypersurface of second type  $\mathcal{M}^{II}$.
Let $c: l = l(v) =  l^{II}(w^1(v),w^2(v))$ be a smooth curve on the hyperbolic sphere $\mathbb{H}^2_1(-1)$, where $w^1 = w^1(v)$, $w^2=w^2(v), \,\, v \in J, \, J \subset \R$.  We consider the two-dimensional
 surface $\mathcal{M}''$ lying on $\mathcal{M}^{II}$ and defined by:
\begin{equation}  \label{E:Eq-4}
\mathcal{M}'': z(u,v) = f(u)\, l(v) + g(u) \,e_4, \quad u \in I, \, v \in J.
\end{equation}
The surface $\mathcal{M}''$, defined by \eqref{E:Eq-4}, is a
one-parameter system of meridians of $\mathcal{M}^{II}$, so we call
it a  \emph{meridian surface on $\mathcal{M}^{II}$}.

The tangent space of $\mathcal{M}''$ is spanned by the vector fields
$$z_u = f'(u) \,l(v) + g'(u) \,e_4; \qquad  z_v = f(u) \,l'(v)$$
and the coefficients of the first fundamental form of $\mathcal{M}''$  are
$$E = \langle z_u, z_u \rangle =  -(f'^2 + g'^2); \quad F = \langle z_u, z_v \rangle = 0; \quad  G = \langle z_v, z_v \rangle = f^2 \langle l',l'\rangle.$$
Since $c$ is a curve lying on $\mathbb{H}^2_1(-1)$, we have $\langle l,l\rangle = -1$, so $t = l'$ satisfies $\langle t, t \rangle = 1$.
We suppose that $f'^2 + g'^2 =1$. Hence, the coefficients of the first fundamental form of $\mathcal{M}''$  are
$E = -1; \, F =  0; \,  G = f^2$.

Now, we have an  orthonormal frame field $\{l(v), t(v), n(v)\}$ of $c$ satisfying the conditions
$\langle l, l \rangle = -1$, $\langle t, t \rangle = 1$,
$\langle n, n \rangle = 1$, and the  following Frenet formulas of $c$ on $\mathbb{H}^2_1(-1)$ hold true:
\begin{equation} \label{E:Eq-2-c}
\begin{array}{l}
\vspace{2mm}
l' = t;\\
\vspace{2mm}
t' =  \kappa \,n + l;\\
\vspace{2mm} n' = - \kappa \,t,
\end{array}
\end{equation}
 where $\kappa (v)= \langle t'(v), n(v) \rangle$ is the spherical curvature of $c$ on  $\mathbb{H}^2_1(-1)$.

We consider the following orthonormal frame field of  $\mathcal{M}''$:
$$X = z_u; \quad Y = \frac{z_v}{f} = t; \quad n_1 = n(v); \quad n_2 = -g'(u)\,l(v) + f'(u) \, e_4.$$
This frame field satisfies $\langle X, X \rangle = -1$,
$\langle Y, Y \rangle = 1$, $\langle X, Y \rangle =0$, $\langle n_1, n_1 \rangle = 1$,
$\langle n_2, n_2 \rangle = -1$, $\langle n_1, n_2 \rangle =0$.
Using  \eqref{E:Eq-2-c} we get:
\begin{equation*} \label{E:Eq-3-c}
\begin{array}{l}
\vspace{2mm}
h(X,X) = \kappa_m\,n_2; \\
\vspace{2mm}
h(X,Y) =  0; \\
\vspace{2mm}
h(Y,Y) =  \frac{\kappa}{f}\,n_1 - \frac{g'}{f} \, n_2,
\end{array}
\end{equation*}
where $\kappa_m = f' g'' -  f'' g'$. The Gauss curvature $K$ and the normal mean curvature vector field $H$ of $\mathcal{M}''$ are given, respectively  by
\begin{equation*} \label{E:Eq-Kc}
K=\frac{f''}{f};
\end{equation*}
\begin{equation} \label{E:Eq-Hc}
H = \frac{\kappa}{2f} \,n_1 + \frac{f f''+(f')^2 - 1}{2f \sqrt{1-f'^2}} \,n_2.
\end{equation}

Note that on the rotational hypersurface $\mathcal{M}^I$  we can consider two types of Lorentz meridian surfaces, namely surfaces of type $\mathcal{M}'_a$ and $\mathcal{M}'_b$,  while on $\mathcal{M}^{II}$ we can construct only one type of Lorentz meridian surfaces, namely $\mathcal{M}''$.

\subsection{Lorentz meridian surfaces  lying on a rotational hypersurface with spacelike axis}

In this subsection we shall explain the construction of  meridian surfaces lying on a
 rotational hypersurface with spacelike axis $Oe_1$. Similarly, we can consider meridian surfaces lying on a rotational hypersurface with axis $Oe_2$.

Let $\E^3_2$ be the Minkowski space $\E^3_2 = \span\left \{ e_{2},e_{3},e_{4}\right \}$. In $\E^3_2$ we can consider two types of spheres, namely the  de Sitter space $\mathbb{S}^2_2(1) =\left\{V\in \E^3_2: \langle V, V \rangle = 1\right \}$,  and the hyperbolic space $\mathbb{H}^2_1(-1)=\left\{V \in  \E^3_2: \langle V, V \rangle = - 1\right \}$.  So, we can consider two types of rotational hypersurfaces about the axis  $Oe_1$.

Let  $f=f(u)$, $g=g(u)$ be smooth functions, defined in an interval $I\subset \R$, and $\tilde{l}^I(w^1,w^2) = \cosh w^1 \,e_2 +   \sinh w^1 \cos w^2 \,e_3 + \sinh w^1 \sin w^2 \,e_4$ be the unit position vector of the sphere
$\mathbb{S}^2_2(1)$ in $\E^3_2 = \span\left \{ e_{2},e_{3},e_{4}\right \}$ centered at the origin $O$.
The first type rotational
hypersurface $\widetilde{\mathcal{M}}^I$, obtained by the rotation of the
meridian curve $m:u\rightarrow (f(u),g(u))$ about the axis $Oe_1$, is
parametrized as follows:
\begin{equation*}
\widetilde{\mathcal{M}}^I:Z(u,w^1,w^2) = g(u) \,e_1 + f(u)\left(\cosh w^1 \,e_2 +   \sinh w^1 \cos w^2 \,e_3 + \sinh w^1 \sin w^2 \,e_4\right),
\end{equation*}
or equivalently,
\begin{equation*}
\widetilde{\mathcal{M}}^I:Z(u,w^1,w^2) =  g(u) \,e_1 + f(u) \,\tilde{l}^I(w^1,w^2).
\end{equation*}

The second type rotational
hypersurface $\widetilde{\mathcal{M}}^{II}$, obtained by the rotation of the
meridian curve $m$ about $Oe_1$, is parametrized  as follows:
\begin{equation*}
\widetilde{\mathcal{M}}^{II}:Z(u,w^1,w^2) = g(u) \,e_1 + f(u)\left(\sinh w^1 \,e_2 + \cosh w^1 \cos w^2 \,e_3 + \cosh w^1 \sin w^2 \,e_4\right),
\end{equation*}
or equivalently,
\begin{equation*}
\widetilde{\mathcal{M}}^{II}:Z(u,w^1,w^2) =  g(u) \,e_1 + f(u) \,\tilde{l}^{II}(w^1,w^2),
\end{equation*}
where
$\tilde{l}^{II}(w^1,w^2) = \sinh w^1 \,e_2 + \cosh w^1 \cos w^2 \,e_3 + \cosh w^1 \sin w^2 \,e_4$ is the unit position vector of the hyperbolic sphere
$\mathbb{H}^2_1(-1)$ in $\E^3_2 = \span\left \{ e_{2},e_{3},e_{4}\right \}$ centered at the origin $O$.

Now, we shall consider Lorentz surfaces in $\E^4_2$ which are one-parameter systems of meridians of the rotational hypersurface $\widetilde{\mathcal{M}}^I$ or $\widetilde{\mathcal{M}}^{II}$.

\vskip 3mm
\textbf{Meridian surfaces on  $\widetilde{\mathcal{M}}^I$}:

Let  $c: l = l(v) =  \tilde{l}^I(w^1(v),w^2(v)), \, v \in J, \, J \subset \R$ be a smooth curve on  $\mathbb{S}^2_2(1)$.
We consider the two-dimensional
 surface $\widetilde{\mathcal{M}}'$ lying on $\widetilde{\mathcal{M}}^I$ and defined by:
\begin{equation}  \label{E:Eq-1-tilde}
\widetilde{\mathcal{M}}': z(u,v) = g(u) \,e_1 + f(u)\, l(v), \quad u \in I, \, v \in J.
\end{equation}
The surface $\widetilde{\mathcal{M}}'$ is a
one-parameter system of meridians of $\widetilde{\mathcal{M}}^I$. It can easily be seen that the surface 
 $\mathcal{M}''$, defined by \eqref{E:Eq-4}, can be transformed into the surface $\widetilde{\mathcal{M}}'$ by the transformation $T$ given by 
\begin{equation} \label{E:Eq-T}
T = \left(\begin{array}{cccc}
\vspace{2mm}
0 & 0& 0&1\\
\vspace{2mm}
0 & 0& 1&0\\
\vspace{2mm}
1 & 0& 0&0\\
\vspace{2mm}
 0& 1& 0&0\\
\end{array}\right).
\end{equation}
So, the meridian surfaces  $\mathcal{M}''$ and  $\widetilde{\mathcal{M}}'$ are congruent. Hence, all results concerning the surface  $\mathcal{M}''$  hold true for the surface $\widetilde{\mathcal{M}}'$.

\vskip 3mm
\textbf{Meridian surfaces on  $\widetilde{\mathcal{M}}^{II}$}:

Now we shall consider meridian surfaces lying on the second type rotational hypersurface  $\widetilde{\mathcal{M}}^{II}$.

Let $c: l = l(v) =  \tilde{l}^{II}(w^1(v),w^2(v)), \, v \in J, \, J \subset \R$ be a smooth curve on the hyperbolic sphere $\mathbb{H}^2_1(-1)$ in $\E^3_2$.  We consider the meridian surface $\widetilde{\mathcal{M}}''$ lying on $\widetilde{\mathcal{M}}^{II}$ and defined by:
\begin{equation}  \label{E:Eq-4-tilde}
\widetilde{\mathcal{M}}'': z(u,v) = g(u) \,e_1 + f(u)\, l(v), \quad u \in I, \, v \in J.
\end{equation}

The tangent space of $\widetilde{\mathcal{M}}''$  is spanned by the vector fields
$$z_u = g'(u) \,e_1 + f'(u) \,l(v); \qquad  z_v = f(u) \,l'(v),$$
so, the coefficients of the first fundamental form  are
$$E = g'^2  -  f'^2; \quad F = 0; \quad  G = f^2 \langle l',l'\rangle.$$
Now,  we consider the following  two cases:

\vskip 2mm
\textbf{Case (a)}: Let  $c$ be a spacelike curve, i.e. $\langle l',l'\rangle =1$.
In this case we suppose that  $f'^2 - g'^2 = 1$. Then the coefficients of the first fundamental form are $E = -1; \, F = 0; \, G = f^2$.
We shall denote the meridian surface in this case by $\widetilde{\mathcal{M}}''_a$.
Under the tranformation $T$ given by \eqref{E:Eq-T} the surface $\mathcal{M}'_b$ is transformed into the surface $\widetilde{\mathcal{M}}''_a$

\vskip 2mm
\textbf{Case (b)}: Let the curve $c$ be timelike, i.e.  $\langle l',l'\rangle = -1$. In this case we assume that $f'^2 - g'^2 = -1$.
Then for the coefficients of the first fundamental form we have $E = 1; \, F = 0$; $G = - f^2$.
We  denote the meridian surface in this case by $\widetilde{\mathcal{M}}''_b$. It is clear that the meridian surfaces $\mathcal{M}'_a$ and 
$\widetilde{\mathcal{M}}''_b$ are congruent (up to the transformation $T$).

\vskip 2mm
In the present paper we will study three types of Lorentz meridian surfaces in $\E^4_2$, namely the surfaces denoted by $\mathcal{M}'_a$, $\mathcal{M}'_b$, and $\mathcal{M}''$.

\section{Classification of meridian surfaces  with parallel mean curvature vector field}

In this section we shall describe all meridian surfaces defined in the previous section which have parallel mean curvature vector field.

\vskip 2mm
First we consider the meridian surface  $\mathcal{M}'_a$, defined by \eqref{E:Eq-1}, where  $f'^2 - g'^2 = -1$. Using formulas \eqref{E:Eq-2} and \eqref{E:Eq-2-2}, we get
\begin{equation} \label{E:Eq-5}
\begin{array}{ll}
\vspace{2mm}
\widetilde{\nabla}_X n_1 =0; & \qquad \widetilde{\nabla}_X n_2 = \kappa_m X;\\
\vspace{2mm}
\widetilde{\nabla}_Y n_1 = - \frac{\kappa}{f} \, Y; & \qquad \widetilde{\nabla}_Y n_2 = \frac{g'}{f} \, Y.
\end{array}
\end{equation}
The mean curvature vector field $H$ of the meridian surface $\mathcal{M}'_a$ is given by formula  \eqref{E:Eq-Ha}. Hence, by use of \eqref{E:Eq-5} we obtain
\begin{equation} \label{E:Eq-6}
\begin{array}{l}
\vspace{2mm}
\widetilde{\nabla}_X H = - \frac{f'' (f f''+(f')^2 + 1)}{2f (f'^2+1)} \, X +   \frac{\kappa f'}{2f^2} \,n_1 - \left(\frac{f f''+(f')^2 + 1}{2f g'} \right)' \,n_2;\\
\vspace{2mm}
\widetilde{\nabla}_Y H = \frac{\kappa^2 - (f f''+(f')^2 + 1)}{2f^2} \, Y -  \frac{\kappa'}{2f^2} \,n_1.
\end{array}
\end{equation}

\begin{thm} \label{T:parallel-H-a}
Let  $\mathcal{M}'_a$ be a meridian surface on $\mathcal{M}^I$ defined by \eqref{E:Eq-1}, (resp. $\widetilde{\mathcal{M}}''_b$ be a meridian surface on  $\widetilde{\mathcal{M}}^{II}$ defined by  \eqref{E:Eq-4-tilde}),
where  $f'^2 - g'^2 = -1$. Then
$\mathcal{M}'_a$ (resp.  $\widetilde{\mathcal{M}}''_b$ ) has parallel mean curvature vector field if and
only if one of the following cases holds:

\hskip 10mm (i) the curve $c$   has constant spherical curvature
and the  meridian $m$ is defined by
$f(u) = a$, $g(u) = \pm u + b$, where $a = const \neq 0$, $b=const$.
In this case  $\mathcal{M}'_a$ (resp.  $\widetilde{\mathcal{M}}''_b$ ) is a flat CMC-surface.

\hskip 10mm (ii) the curve $c$   has zero spherical curvature  and the meridian $m$ is determined by $f' = \varphi(f)$ where
\begin{equation} \notag
\varphi(t) = \pm \frac{1}{t} \sqrt{(c \pm a\, t^2)^2 -t^2}, \quad a =
const \neq 0, \quad c = const,
\end{equation}
$g(u)$ is defined by $g' = \sqrt{f'^2 + 1}$. In this case  $\mathcal{M}'_a$ (resp.  $\widetilde{\mathcal{M}}''_b$ ) lies in a  hyperplane of $\E^4_2$.
\end{thm}

\begin{proof}

Let  $\mathcal{M}'_a$ be a surface with parallel mean curvature vector field. Using formulas
\eqref{E:Eq-6} we get the following conditions
\begin{equation} \label{E:Eq-7}
\begin{array}{l}
\vspace{2mm}
  \kappa' =0;\\
 \vspace{2mm}
\kappa f' =0;\\
\vspace{2mm}
\left(\frac{f f''+(f')^2 + 1}{2f g'} \right)' =0.
\end{array}
\end{equation}
The first equality of  \eqref{E:Eq-7} implies that the spherical curvature $\kappa$ of $c$ is constant.
Having in mind \eqref{E:Eq-7} we obtain that  there are  two possible  cases:

\vskip 1mm
Case (i): $f' = 0$, i.e. $f(u) = a$, $a = const \neq 0$. Using that $f'^2 - g'^2 = -1$, we get $g(u) = \pm u+b$, $b = const$.
In this case the mean curvature vector field is expressed as follows
$$H = -\frac{\kappa}{2a}\, n_1 \mp \frac{1}{2a} \, n_2.$$
The last equality implies that $\langle H, H \rangle = \frac{1-\kappa^2}{4a^2}=const$. Hence,  $\mathcal{M}'_a$  has constant mean curvature.
If $\kappa^2 =1$, then  $\mathcal{M}'_a$ is quasi-minimal. If $\kappa^2 \neq 1$, then  $\mathcal{M}'_a$ has non-zero constant mean curvature.
Having in mind that the  Gauss curvature of $\mathcal{M}'_a$ is expressed by formula \eqref{E:Eq-Ka}, in this case we obtain $K=0$, i.e. $\mathcal{M}'_a$ is flat.

\vskip 1mm
Case (ii): $\kappa = 0$ and $\frac{f f''+(f')^2 + 1}{2f g'}  = a = const$.
It follows from \eqref{E:Eq-5} that in the case  $\kappa =0$ we have  $\widetilde{\nabla}_X n_1 = \widetilde{\nabla}_Y n_1 = 0$,
and hence $\mathcal{M}'_a$ lies in the 3-dimensional constant hyperplane  $span \{X,Y,n_2\}$.
If $a = 0$, then $H=0$, i.e. $\mathcal{M}'_a$ is minimal. Since we consider non-minimal surfaces, we assume that $a \neq 0$.
In this case the meridian $m$ is determined by the following differential
equation:
\begin{equation} \label{E:Eq-8}
f f''+(f')^2 + 1= \pm  2a f \sqrt{f'^2 + 1}, \qquad a = const \neq 0.
\end{equation}
The solutions of the above  differential equation  can be found in the  following way.
Setting $f' = \varphi (f)$ in equation \eqref{E:Eq-8}, we obtain
that the function $\varphi  = \varphi (t)$ is a solution of the equation:
\begin{equation} \label{E:Eq-9}
\frac{t}{2} \,(\varphi ^2)' + \varphi ^2 + 1 = \pm 2 a t \sqrt{\varphi ^2 + 1}.
\end{equation}
If we  set $z(t) = \sqrt{\varphi ^2(t) +1}$, equation \eqref{E:Eq-9} takes the form
\begin{equation} \notag
z' + \frac{1}{t}\, z = \pm 2a.
\end{equation}
The general solution of the last equation is given by the formula $z(t) = \frac{c \pm at^2}{t}$, $ c = const$.
Hence, the general solution of \eqref{E:Eq-9} is
\begin{equation} \notag
\varphi(t) = \pm \frac{1}{t} \sqrt{(c \pm a\, t^2)^2 -t^2}.
\end{equation}

Conversely, if one of the cases (i) or (ii) stated in the theorem  holds true, then by direct computation we get that
$D_X H = D_Y H = 0$, i.e. the surface has parallel mean curvature vector field.

\end{proof}

\vskip 2mm
Next, we consider the meridian surface  $\mathcal{M}'_b$, defined by \eqref{E:Eq-1}, where  $f'^2 - g'^2 = 1$.
The mean curvature vector field  of $\mathcal{M}'_b$ is given by formula  \eqref{E:Eq-Hb}.
Similarly to the considerations about the meridian surface  $\mathcal{M}'_a$, now we obtain
\begin{equation} \label{E:Eq-6-b}
\begin{array}{l}
\vspace{2mm}
\widetilde{\nabla}_X H = \frac{f'' (f f''+(f')^2 - 1)}{2f (f'^2-1)} \, X +   \frac{\kappa f'}{2f^2} \,n_1 + \left(\frac{f f''+(f')^2 - 1}{2f g'} \right)' \,n_2;\\
\vspace{2mm}
\widetilde{\nabla}_Y H = \frac{f f''+(f')^2 - 1 - \kappa^2}{2f^2} \, Y -  \frac{\kappa'}{2f^2} \,n_1.
\end{array}
\end{equation}

In the following theorem we give the classification of the  meridian surfaces of type $\mathcal{M}'_b$ having parallel mean curvature vector field.

\begin{thm} \label{T:parallel-H-b}
Let  $\mathcal{M}'_b$ be a meridian surface on $\mathcal{M}^I$
defined by \eqref{E:Eq-1} (resp. $\widetilde{\mathcal{M}}''_a$ be a meridian surface on  $\widetilde{\mathcal{M}}^{II}$ defined by  \eqref{E:Eq-4-tilde}), where  $f'^2 - g'^2 = 1$. Then
$\mathcal{M}'_b$  (resp. $\widetilde{\mathcal{M}}''_a$)  has parallel mean curvature vector field if and
only if the curve $c$   has zero spherical curvature  and the meridian $m$ is determined by $f' = \varphi(f)$ where
\begin{equation} \notag
\varphi(t) = \pm \frac{1}{t} \sqrt{(c \pm a\, t^2)^2 +t^2}, \quad a =
const \neq 0, \quad c = const,
\end{equation}
$g(u)$ is defined by $g' = \sqrt{f'^2 - 1}$. Moreover,  $\mathcal{M}'_b$  (resp. $\widetilde{\mathcal{M}}''_a$) lies in a  hyperplane of $\E^4_2$.
\end{thm}

\begin{proof}

Let  $\mathcal{M}'_b$ be a surface with parallel mean curvature vector field. Formulas
\eqref{E:Eq-6-b} imply the following conditions
\begin{equation*} \label{E:Eq-7-b}
\begin{array}{l}
\vspace{2mm}
  \kappa' =0;\\
 \vspace{2mm}
\kappa f' =0;\\
\vspace{2mm}
\left(\frac{f f''+(f')^2 - 1}{2f g'} \right)' =0,
\end{array}
\end{equation*}
and hence, we get  $\kappa = const$.
If we assume that  $f' = 0$, i.e. $f(u) = a = const$, then having in mind that $f'^2 - g'^2 = 1$, we get $g'^2 = -1$, which is not possible.
So, the only possible case is  $\kappa = 0$ and $\frac{f f''+(f')^2 - 1}{2f g'}  = a = const$.
Since  $\kappa =0$, we have  $\widetilde{\nabla}_X n_1 = \widetilde{\nabla}_Y n_1 = 0$.
So, $\mathcal{M}'_b$ lies in the 3-dimensional constant hyperplane  $span \{X,Y,n_2\}$ of $\E^4_2$.
We consider non-minimal surfaces, so we assume that $a \neq 0$.
The meridian $m$ is determined by the following differential
equation:
\begin{equation} \label{E:Eq-8-b}
f f''+(f')^2 - 1= \pm  2a f \sqrt{f'^2 - 1}, \qquad a = const \neq 0.
\end{equation}
Similarly to the proof of Theorem \ref{T:parallel-H-a}, setting $f' = \varphi (f)$ in equation \eqref{E:Eq-8-b},  we obtain
\begin{equation} \label{E:Eq-10}
\varphi(t) = \pm \frac{1}{t} \sqrt{(c \pm a\, t^2)^2 +t^2}, \quad a = const \neq 0, \quad c = const.
\end{equation}

Conversely, if $\kappa = 0$ and the meridian $m$ is determined by \eqref{E:Eq-10}, then direct computation show that
$D_X H = D_Y H = 0$, i.e. $\mathcal{M}'_b$  has parallel mean curvature vector field.

\end{proof}

\vskip 2mm
Now, let us consider the meridian surface  $\mathcal{M}''$, defined by  \eqref{E:Eq-4}, where $f'^2 + g'^2 =1$.
The mean curvature vector field  of  $\mathcal{M}''$ is given by formula  \eqref{E:Eq-Hc}.
The derivatives of $H$ with respect to $X$ and $Y$ are given by the following formulas
\begin{equation*} \label{E:Eq-6-c}
\begin{array}{l}
\vspace{2mm}
\widetilde{\nabla}_X H = \frac{f'' (f f''+(f')^2 - 1)}{2f (1-f'^2)} \, X - \frac{\kappa f'}{2f^2} \,n_1 + \left(\frac{f f''+(f')^2 - 1}{2f g'} \right)' \,n_2;\\
\vspace{2mm}
\widetilde{\nabla}_Y H = -\frac{f f''+(f')^2 - 1 - \kappa^2}{2f^2} \, Y + \frac{\kappa'}{2f^2} \,n_1.
\end{array}
\end{equation*}

Similarly to the proof of  Theorem \ref{T:parallel-H-a}, we obtain the following classification result.

\begin{thm} \label{T:parallel-H-c}
Let  $\mathcal{M}''$ be a meridian surface on $\mathcal{M}^{II}$
defined by \eqref{E:Eq-4} (resp. $\widetilde{\mathcal{M}}'$ be a meridian surface on $\widetilde{\mathcal{M}}^{I}$,
defined by \eqref{E:Eq-1-tilde}), where  $f'^2 + g'^2 = 1$. Then
$\mathcal{M}''$ (resp. $\widetilde{\mathcal{M}}'$) has parallel mean curvature vector field if and
only if one of the following cases holds:

\hskip 10mm (i) the curve $c$   has constant spherical curvature
and the  meridian $m$ is defined by
$f(u) = a$, $g(u) = \pm u + b$, where $a = const \neq 0$, $b=const$.
In this case  $\mathcal{M}''$ (resp. $\widetilde{\mathcal{M}}'$) is a flat CMC-surface.

\hskip 10mm (ii) the curve $c$   has zero spherical curvature  and the meridian $m$ is determined by $f' = \varphi(f)$ where
\begin{equation} \notag
\varphi(t) = \pm \frac{1}{t} \sqrt{t^2 -(c \pm a\, t^2)^2 }, \quad a =
const \neq 0, \quad c = const,
\end{equation}
$g(u)$ is defined by $g' = \sqrt{1-f'^2}$. In this case  $\mathcal{M}''$ (resp. $\widetilde{\mathcal{M}}'$) lies in a  hyperplane of $\E^4_2$.
\end{thm}

\section{Classification of meridian surfaces  with parallel normalized  mean curvature vector field}

In this section we give the classification of all meridian surfaces  which have parallel normalized  mean curvature vector field but not parallel $H$.

\vskip 2mm
First we consider the meridian surface  $\mathcal{M}'_a$, defined by \eqref{E:Eq-1}, where  $f'^2 - g'^2 = -1$. The mean curvature vector field $H$ of $\mathcal{M}'_a$ is given by formula  \eqref{E:Eq-Ha}. We assume that $\langle H,H \rangle \neq 0$, i.e. $(f f''+(f')^2 + 1)^2 - \kappa^2 (f'^2+1) \neq 0$, and denote $\varepsilon = sign \langle H,H \rangle$. Then the  normalized mean curvature vector field of $\mathcal{M}'_a$ is given by 
\begin{equation} \label{E:Eq-H0-a}
H_0 =  \frac{1}{\sqrt{\varepsilon\left((f f''+(f')^2 + 1)^2 - \kappa^2 (f'^2+1)\right)}} \left(-\kappa \sqrt{f'^2+1} \,n_1 - (f f''+(f')^2 + 1) \,n_2\right).
\end{equation}

If $\kappa = 0$, then $H_0 = n_2$ and \eqref{E:Eq-5} implies that $D_X H_0 = D_Y H_0 =0$, i.e. $H_0$ is parallel in the normal bundle. We  consider this case as trivial, since under the assumption $\kappa =0$ the surface $\mathcal{M}'_a$ lies in a 3-dimensional space $\E^3_1$ and every surface in $\E^3_1$ has parallel normalized mean curvature vector field. So, further we assume that $\kappa \neq 0$.

For simplicity we denote $$A = \frac{-\kappa \sqrt{f'^2+1}}{\sqrt{\varepsilon\left((f f''+(f')^2 + 1)^2 - \kappa^2 (f'^2+1)\right)}},  
\quad B = \frac{-(f f''+(f')^2 + 1)}{\sqrt{\varepsilon\left((f f''+(f')^2 + 1)^2 - \kappa^2 (f'^2+1)\right)}},$$
so, the normalized mean  curvature vector field is  expressed as $H_0 = A\, n_1 + B \, n_2$. Then equalities \eqref{E:Eq-H0-a} and \eqref{E:Eq-5} imply 
\begin{equation} \label{E:Eq-11}
\begin{array}{l}
\vspace{2mm}
\widetilde{\nabla}_X H_0 = B \kappa_m \, X +  A'_u\,n_1 + B'_u \,n_2;\\
\vspace{2mm}
\widetilde{\nabla}_Y H_0 = \frac{B g' -A \kappa}{f} \, Y +  \frac{A'_v}{f} \,n_1 +  \frac{B'_v}{f} \,n_2,
\end{array}
\end{equation}
where $A'_u$ (resp. $A'_v$) denotes $\frac{\partial A}{\partial u}$ (resp. $\frac{\partial A}{\partial v}$).

\begin{thm} \label{T:parallel-H0-a}
Let  $\mathcal{M}'_a$ be a meridian surface on $\mathcal{M}^I$ defined by \eqref{E:Eq-1}, (resp. $\widetilde{\mathcal{M}}''_b$ be a meridian surface on  $\widetilde{\mathcal{M}}^{II}$ defined by  \eqref{E:Eq-4-tilde}),
where  $f'^2 - g'^2 = -1$. Then
$\mathcal{M}'_a$ (resp.  $\widetilde{\mathcal{M}}''_b$ ) 
has parallel normalized mean curvature vector field but not parallel mean curvature vector if and
only if one of the following cases holds:

\hskip 10mm (i) $\kappa \neq 0$ and the  meridian $m$ is defined by
$$f(u) = \pm \sqrt{-u^2 +2au+b}, \quad g(u) = \pm \sqrt{a^2+b} \arcsin \frac{u-a}{\sqrt{a^2+b}} + c,$$ 
where $a = const$, $b=const$, $c = const$.

\hskip 10mm (ii) the curve $c$   has non-zero constant spherical curvature  and the meridian $m$ is determined by $f' = \varphi(f)$ where
\begin{equation} \notag
\varphi(t) = \pm \frac{1}{t} \sqrt{(c t+a)^2 -t^2}, \quad a =
const, \; c = const \neq 0,\; c^2 \neq \kappa^2,
\end{equation}
$g(u)$ is defined by $g' = \sqrt{f'^2 + 1}$. 
\end{thm}

\begin{proof}
Let $\mathcal{M}'_a$ be a surface with parallel normalized mean curvature vector field, i.e. $D_X H_0 = D_Y H_0 =0$. 
Then from  \eqref{E:Eq-11} it follows that  $A=const$, $B=const$. Hence, 
\begin{equation} \label{E:Eq-12}
\begin{array}{l}
\vspace{2mm}
\frac{-\kappa \sqrt{f'^2+1}}{\sqrt{\varepsilon\left((f f''+(f')^2 + 1)^2 - \kappa^2 (f'^2+1)\right)}} = \alpha=const;\\
\frac{-(f f''+(f')^2 + 1)}{\sqrt{\varepsilon\left((f f''+(f')^2 + 1)^2 - \kappa^2 (f'^2+1)\right)}} = \beta=const.
\end{array}
\end{equation}

\vskip 1mm
Case (i):  $f f''+(f')^2 + 1=0$. In this case, from \eqref{E:Eq-H0-a} we get that the normalized mean curvature vector field is $H_0 = n_1$ and the mean curvature vector field is $H = - \frac{\kappa}{2f} \,n_1$. Since we study surfaces with $\langle H,H \rangle \neq 0$, we get $\kappa \neq 0$. The solution of the differential equation $f f''+(f')^2 + 1=0$ is given by the formula 
$f(u) = \pm \sqrt{-u^2 +2au+b}$, where $a=const$, $b =const$. Using that $g' = \sqrt{f'^2 + 1}$, we obtain the following equation for $g(u)$:
$$g' = \pm \frac{\sqrt{a^2+b}}{\sqrt{-u^2 +2au+b}}.$$
Integrating the above equation we get
$$g(u) = \pm \sqrt{a^2+b} \arcsin \frac{u-a}{\sqrt{a^2+b}} + c, \quad c = const.$$

\vskip 1mm
Case (ii):  $f f''+(f')^2 + 1 \neq 0$. From \eqref{E:Eq-12} we get
\begin{equation} \label{E:Eq-13}
\frac{\beta}{\alpha} \,\kappa =  \frac{f f''+(f')^2 + 1}{\sqrt{f'^2+1}},  \quad \alpha \neq 0, \; \beta \neq 0.
\end{equation}
Since the left-hand side of equality \eqref{E:Eq-13} is a function of $v$, the right-hand side of \eqref{E:Eq-13} is a function of $u$, we obtain that 
\begin{equation} \notag
\begin{array}{l}
\vspace{2mm}
\frac{f f''+(f')^2 + 1}{\sqrt{f'^2+1}} = c, \quad c = const \neq 0;\\
\kappa = \frac{\alpha}{\beta}\, c = const.
\end{array}
\end{equation}
Then the length of the mean curvature vector field is $\langle H,H \rangle  = \frac{c^2 -\kappa^2}{4f^2}$. 
Since we study surfaces with $\langle H,H \rangle \neq 0$, we get $c^2 \neq \kappa^2$.
The meridian $m$ is determined by the following differential
equation:
\begin{equation} \label{E:Eq-14}
f f''+(f')^2 + 1= c \sqrt{f'^2 + 1}.
\end{equation}
The solutions of this  differential equation  can be found as follows.
We set $f' = \varphi (f)$ in equation \eqref{E:Eq-14} and obtain
that the function $\varphi  = \varphi (t)$ satisfies
\begin{equation} \label{E:Eq-15}
\frac{t}{2} \,(\varphi ^2)' + \varphi ^2 + 1 = c \sqrt{\varphi ^2 + 1}.
\end{equation}
Putting $z(t) = \sqrt{\varphi ^2(t) +1}$, equation \eqref{E:Eq-15} can be  written as
\begin{equation} \notag
z' + \frac{1}{t}\, z = \frac{c}{t},
\end{equation}
whose general solution  is  $z(t) = \frac{ct + a}{t}$, $a = const$.
Hence, the general solution of \eqref{E:Eq-15} is given by the formula
\begin{equation} \notag
\varphi(t) = \pm \frac{1}{t} \sqrt{(c t+a)^2 -t^2}.
\end{equation}

Conversely, if one of the cases (i) or (ii) stated in the theorem  holds true, then by direct computation we get that
$D_X H_0 = D_Y H_0 = 0$, i.e. the surface has parallel normalized  mean curvature vector field.
Moreover, in case (i) we have
$$D_XH = \frac{\kappa f'}{2f^2} \,n_1; \qquad D_YH = -  \frac{\kappa'}{2f^2} \,n_1,$$
which implies that $H$ is not parallel in the normal bundle, since $\kappa \neq 0$, $f' \neq 0$. 
In case (ii) we get
$$D_XH = \frac{\kappa f'}{2f^2} \,n_1 + \frac{cf'}{2f^2} \,n_2; \qquad D_YH = 0,$$
and again we have that $H$ is not parallel in the normal bundle. 
\end{proof}

\vskip 2mm
In a similar way we consider the meridian surface  $\mathcal{M}'_b$, defined by \eqref{E:Eq-1}, where  $f'^2 - g'^2 = 1$. 
The  normalized mean curvature vector field of $\mathcal{M}'_b$ is given by 
\begin{equation*} \label{E:Eq-H0-b}
H_0 =  \frac{1}{\sqrt{\varepsilon\left( \kappa^2 (f'^2-1) - (f f''+(f')^2 + 1)^2\right)}} \left(-\kappa \sqrt{f'^2-1} \,n_1 + (f f''+(f')^2 - 1) \,n_2\right),
\end{equation*}
where  $\varepsilon = sign \langle H,H \rangle$ and we assume that $\kappa^2 (f'^2-1) - (f f''+(f')^2 - 1)^2 \neq 0$.

The classification of the meridian surfaces of type $\mathcal{M}'_b$ and $\widetilde{\mathcal{M}}''_a$ which have parallel normalized  mean curvature vector field but not parallel $H$ is given in the following theorem.

\begin{thm} \label{T:parallel-H0-b}
Let  $\mathcal{M}'_b$ be a meridian surface on $\mathcal{M}^I$
defined by \eqref{E:Eq-1} (resp. $\widetilde{\mathcal{M}}''_a$ be a meridian surface on  $\widetilde{\mathcal{M}}^{II}$ defined by  \eqref{E:Eq-4-tilde}), where  $f'^2 - g'^2 = 1$. Then
$\mathcal{M}'_b$  (resp. $\widetilde{\mathcal{M}}''_a$)  has parallel normalized mean curvature vector field but not parallel mean curvature vector if and
only if one of the following cases holds:

\hskip 10mm (i) $\kappa \neq 0$ and the  meridian $m$ is defined by
$$f(u) = \pm \sqrt{u^2 +2au+b}, \quad g(u) = \pm \sqrt{a^2-b} \ln |u+a + \sqrt{u^2 +2au+b}| + c,$$ 
where $a = const$, $b=const$, $c = const$, $a^2-b>0$.

\hskip 10mm (ii) the curve $c$   has non-zero constant spherical curvature  and the meridian $m$ is determined by $f' = \varphi(f)$ where
\begin{equation} \notag
\varphi(t) = \pm \frac{1}{t} \sqrt{(c t+a)^2 + t^2}, \quad a =
const, \; c = const \neq 0,\; c^2 \neq \kappa^2,
\end{equation}
$g(u)$ is defined by $g' = \sqrt{f'^2 - 1}$. 
\end{thm}

The proof of this theorem is similar to the proof of Theorem \ref{T:parallel-H0-a}.

\vskip 3mm
The classification of the meridian surfaces of type $\mathcal{M}''$ and $\widetilde{\mathcal{M}}'$ which have parallel normalized  mean curvature vector field but not parallel $H$ is given in the following theorem.

\begin{thm} \label{T:parallel-H0-c}
Let  $\mathcal{M}''$ be a meridian surface on $\mathcal{M}^{II}$
defined by \eqref{E:Eq-4} (resp. $\widetilde{\mathcal{M}}'$ be a meridian surface on $\widetilde{\mathcal{M}}^{I}$,
defined by \eqref{E:Eq-1-tilde}), where  $f'^2 + g'^2 = 1$. Then
$\mathcal{M}''$ (resp. $\widetilde{\mathcal{M}}'$) has parallel normalized mean curvature vector field but not parallel mean curvature vector if and
only if one of the following cases holds:

\hskip 10mm (i) $\kappa \neq 0$ and the  meridian $m$ is defined by
$$f(u) = \pm \sqrt{u^2 +2au+b}, \quad g(u) = \pm \sqrt{b -a^2} \ln |u+a + \sqrt{u^2 +2au+b}| + c,$$ 
where $a = const$, $b=const$, $c = const$, $b -a^2>0$.

\hskip 10mm (ii) the curve $c$   has non-zero constant spherical curvature  and the meridian $m$ is determined by $f' = \varphi(f)$ where
\begin{equation} \notag
\varphi(t) = \pm \frac{1}{t} \sqrt{t^2 -(a - c t)^2}, \quad a =
const, \; c = const \neq 0,\; c^2 \neq \kappa^2,
\end{equation}
$g(u)$ is defined by $g' = \sqrt{1- f'^2}$. 
\end{thm}

\vskip 2mm
\begin{rem}
All theorems stated in this section give examples of Lorentz surfaces in the pseudo-Euclidean space $\E^4_2$ which have parallel normalized mean curvature vector field   but not parallel mean curvature vector field.
\end{rem}

 \vskip 5mm \textbf{Acknowledgments:}
The second author is partially supported by the National Science Fund,
Ministry of Education and Science of Bulgaria under contract
DFNI-I 02/14.

The paper is prepared
during the first author's visit at the Institute of
Mathematics and Informatics at the Bulgarian Academy of Sciences,
Sofia, Bulgaria  in December 2015.

\end{document}